\documentclass[graybox]{svmult}

\usepackage{newtxtext}
\usepackage{newtxmath}
\usepackage{makeidx}
\usepackage{graphicx}
\usepackage{multicol}
\usepackage{footmisc}

\usepackage[bookmarks=false]{hyperref}

\usepackage[numbers,sort&compress]{natbib}

\usepackage[utf8]{inputenc}
\usepackage{algorithmic,algorithm}

\usepackage[obeyFinal]{todonotes}

\usepackage{MnSymbol}

\newcommand{\R}{\mathbb{R}}
\newcommand{\C}{\mathbb{C}}
\newcommand{\N}{\mathbb{N}}
\newcommand{\SubA}[1][]{\mathbb{S}_\text{#1}}
\newcommand{\SpDev}[1]{\llangle\cdot,\cdot\rrangle_{#1}}
\newcommand{\Krylov}{\mathcal{K}}
\newcommand{\diag}[1]{\operatorname{diag}\left(#1\right)}

\newcommand{\T}[1]{{T_\textsubscript{#1}}}

\usepackage{xifthen}
\newcommand{\dune}[1][]{
  \textsc{Dune}\ifthenelse{\equal{#1}{}}{}{\textsc{-{#1}}}
}

\definecolor{darkgreen}{RGB}{0,127,0}
\usepackage[final]{listings}
\usepackage[quotient-mode=fraction]{siunitx}
\DeclareSIUnit\flop{FLOP}
\DeclareSIUnit\byte{byte}
\DeclareSIUnit\double{double}
\DeclareSIUnit\cycle{cycle}

\usepackage{mathtools}
\mathtoolsset{showonlyrefs}

\begin{document}

\title*{Strategies for the vectorized Block Conjugate Gradients method}
\author{Nils-Arne Dreier \and Christian Engwer}
\institute{Nils-Arne Dreier \at University of Münster,
  Einsteinstraße 62, 48149 Münster, \email{n.dreier@uni-muenster.de}
  \and
  Christian Engwer \at University of Münster, Einsteinstraße 62, 48149 Münster, \email{c.engwer@uni-muenster.de}}
\maketitle

\abstract{
  Block Krylov methods have recently gained a lot of attraction. Due
  to their increased arithmetic intensity they offer a promising way
  to improve performance on modern hardware.
  Recently \citeauthor{frommer2017block} presented a block Krylov framework that
  combines the advantages of block Krylov methods and data parallel methods.
  We review this framework and apply it on the Block Conjugate Gradients method,
  to solve linear systems with multiple right hand sides.
  In this course we consider challenges that occur on modern hardware, like a
  limited memory bandwidth, the use of SIMD instructions and the communication
  overhead.
  We present a performance model to predict the efficiency of
  different Block CG variants and compare these with experimental numerical results.
}


\section{Introduction}
Developers of numerical software are facing multiple
challenges on modern HPC-hardware.
Firstly, multiple levels of concurrency must be exploited to achieve
the full performance.
Secondly, due to that parallelism, communication between nodes is needed, which
must be cleverly organized to avoid an expensive overhead.
And most importantly, modern CPUs have a low memory bandwidth, compared to the peak FLOP
rate, such that for standard linear solvers the memory bandwidth is the bottleneck for the
performance. Therefore only algorithms with a high arithmetic intensity will
perform well.

Instruction level parallelism is now apparent on all modern
CPU architectures. They provide dedicated vector (or SIMD) instructions, that allow to proceed multiple floating point operations with one
instruction call,
e.g. AVX-512 allows processing 8 \lstinline{double} at once.
The efficient use of these instructions is a further challenge.

The Conjugate Gradients method (CG) is a standard tool for solving large,
sparse, symmetric, positive definite, linear systems.
The Block Conjugate Gradient (BCG) method was introduced in the 80s to improve the
convergence rate for systems with multiple right hand sides
\cite{o1980block}.
Recently these methods have been rediscovered to reduce the communication overhead
in parallel computations
\cite{grigori2016enlarged,grigori2017reducing,al2017enlarged}.

In this paper we present a generalization of the BCG method, which makes it
applicable to arbitrary many right-hand-sides.
We consider a symmetric, positive definite matrix $A\in\R^{n\times
  n}$ 
and want to solve the matrix equation
\begin{align}
  AX=B, \qquad\text{with }B,X\in \R^{n\times k}.
\end{align}

This paper is structured as follows.
In Section \ref{sec:blockkrylov} we briefly review the theoretical background
of block Krylov methods, using the notation of \cite{frommer2017block}.
Then in Section \ref{sec:blockcg} we apply this theory on the BCG method.
The implementation of the method, a theoretical performance model and some numerical experiments are presented in
Section \ref{sec:impl_and_numeric}.
\section{Block Krylov subspaces}
\label{sec:blockkrylov}
Considering functions of matrices,  \citeauthor{frommer2017block}
presented in \cite{frommer2017block} a generic framework for block
Krylov methods.
Further work on this framework can be found
in the PhD thesis of
\citeauthor{lund2018block}~\cite{lund2018block}.
In the following we review the most important definitions, which we will need in
section \ref{sec:blockcg}.
\citeauthor[]{frommer2017block} used $\C$ as numeric field,
for simplicity of the following numerics, we restrict our self to $\R$.
\begin{definition}[Block Krylov subspace]
  Let $\SubA$ be a *-subalgebra of $\R^{k\times k}$ and $R\in\R^{n\times k}$.
  The $m$-th block Krylov subspace with respect to $A, R$ and $\SubA$ is defined by
  \begin{align}
    \Krylov^m_{\SubA}(A,R) = \left\{ \sum_{i=0}^{m-1} A^iRc_i \,\big|\, c_0,\ldots,c_{m-1}\in\SubA \right\} \subset \R^{n\times k}.
  \end{align}
\end{definition}
From that definition we find the following lemma directly.
\begin{lemma}
  If $\SubA[1]$ and $\SubA[2]$ are two *-subalgebras of $\R^{n \times n}$, with $\SubA[1] \subseteq \SubA[2]$.
  Then
  \begin{align}
    \Krylov^m_{\SubA[1]}(A,R) \subseteq \Krylov^m_{\SubA[2]}(A,R)
  \end{align}
  holds.
\end{lemma}
In this paper we want to consider the following *-subalgebras and
$\SubA$-products and the corresponding Krylov spaces.

\newpage
\begin{definition}[Relevant *-subalgebras]
  Let $p\in \N$ be a divider of $k$.
  We define the following *-subalgebras and corresponding products:
    \begin{align}
      \text{hybrid: }
      && \SubA[Hy]^p &:= \diag{\left(\R^{p\times p}\right)^q}
      &\Rightarrow&& \SpDev{\SubA[Hy]^p} &= \diag{X_1^*Y_1,\ldots,X_q^*Y_q},\\
      \text{global: }
      && \SubA[Gl]^p &:= \R^{p\times p} \otimes I_q
      &\Rightarrow&& \SpDev{\SubA[Gl]^p} &=\sum_{i=0}^q X_i^*Y_i \otimes I_q,\\
      \intertext{where $I_q$ denotes the $q$ dimensional identity matrix and
      $\diag{\left(\R^{p\times p}\right)^q}$ denotes the set of $k\times k$ matrices
      where only the $p\times p$ diagonal matrices have non-zero values.
      Furthermore we define the special cases
      }
      \text{classical: }
      && \SubA[Cl] &:= \R^{k\times k} = \SubA[Hy]^k = \SubA[Gl]^k &&&\text{and}\\
      \text{parallel: }
      && \SubA[Pl] &:= \diag{\R^k} = \SubA[Hy]^1.
    \end{align}
    The names result from the behavior of the resulting Krylov method; $\SubA[Cl]$
    yields in the classical block Krylov method as presented by
    \citeauthor{o1980block}~\cite{o1980block}, whereas $\SubA[Pl]$ results in a CG
    method, which is carried out for all right hand sides
    simultaneously, in a instruction level parallel fashion.

\end{definition}
From that definition we could conclude the following embedding lemma.
\begin{lemma}[Embeddings of *-subalgebras]
  For $p_1, p_2\in\mathbb{N}$, where $p_1$ is a divisor of $p_2$ and $p_2$ is a
  divisor of $k$
  , we have the following embedding:
  \begin{align}
    \begin{array}{ccccccc}
      \SubA[Pl] & \subseteq & \SubA[Hy]^{p_1} & \subseteq & \SubA[Hy]^{p_2} & \subseteq & \SubA[Cl]\\
      \rotatebox[origin=c]{90}{$\subset$} && \rotatebox[origin=c]{90}{$\subseteq$} && \rotatebox[origin=c]{90}{$\subseteq$} && \rotatebox[origin=c]{90}{$=$}\\
      \SubA[Gl]^{1} & \subseteq & \SubA[Gl]^{p_1} & \subseteq & \SubA[Gl]^{p_2} & \subseteq & \SubA[Cl]
    \end{array}
  \end{align}
\end{lemma}
\section{Block Conjugate Gradient method}
\label{sec:blockcg}
Algorithm \ref{alg:bcg} shows the preconditioned BCG method.
We recompute $\rho^{i-1}$ in line \ref{ali:stablization} to improve the
stability of the method.
A more elaborate discussion of the stability of the BCG method can be found in
the paper of \citeauthor{dubrulle2001retooling}~\cite{dubrulle2001retooling}.
This stabilization has only mild effect on the performance as the
communication that is needed to compute the block product could be carried out
together with the previous block product.

\begin{algorithm}
  \caption{Preconditioned Block Conjugate Gradients method (stablized)}
  \label{alg:bcg}
  \begin{algorithmic}[1]
    \STATE $R^0 \gets B-AX^0$
    \STATE $P^1 \gets M^{-1}R^0$
    \FOR{$i = 1,\ldots$ \TO convergence}
    \STATE $Q^i \gets AP^i$
    \STATE $\alpha^i \gets \llangle P^i, Q^i\rrangle_{\SubA}$
    \STATE $\rho^{i-1} \gets \llangle P^i, R^{i-1} \rrangle_{\SubA}$ \COMMENT{recompute} \label{ali:stablization}
    \STATE $\lambda^i \gets \left(\alpha^i\right)^{-1}\rho^{i-1}$
    \STATE $X^i \gets X^{i-1} + P^i\lambda^i$
    \STATE $R^i \gets R^{i-1} - Q^i\lambda^i$ \label{algl:residual}
    \STATE $Z^{i+1} \gets M^{-1}R^i$
    \STATE $\rho^i \gets \llangle Z^{i+1}, R^i \rrangle_{\SubA}$
    \STATE $\beta^i \gets {\rho^{i-1}}^{-1}\rho^i$
    \STATE $P^{i+1} \gets Z^{i+1} - P^i\beta^i$
    \ENDFOR
  \end{algorithmic}
\end{algorithm}
\noindent The algorithm is build-up from 4 kernels:
\begin{itemize}
\item \texttt{BDOT}: Computing the block product,
  $\gamma \gets \llangle X, Y \rrangle_{\SubA{}}$
\item \texttt{BAXPY}: Generic vector update $X \gets X + Y\gamma$
\item \texttt{BOP}: Applying the operator (or preconditioner) on a block vector $Y \gets AX$
\item \texttt{BSOLVE}: Solve a block system in the *-subalgebra
  $\delta \gets \gamma^{-1}\delta$
\end{itemize}

O'Leary showed the following convergence result to estimate the error of the
classical BCG method.
\begin{theorem}[Convergence of Block Conjugate Gradients {\cite[Theorem 5]{o1980block}}]
  \label{thm:blockcg}
  For the energy-error of the $s$-th column $\|e^i_{s}\|_A$ of the classical BCG method, the following estimation hold:
  \begin{align*}
    \|e^i_s\|_A &\leq c_1 \mu^i\\
    \text{with } \mu = \frac{\sqrt{\kappa_k}-1}{\sqrt{\kappa_k}+1}, \kappa_k &= \frac{\lambda_n}{\lambda_k} \text{ and constant } c_1,
  \end{align*}
  where $\lambda_1\leq \ldots \leq \lambda_N$ denotes the eigenvalues of the
  preconditioned matrix ${M^{-\frac12}AM^{-\frac12}}$.
  The constant $c_1$ depends on $s$ and the initial error $e^0$ but not on~$i$.  
\end{theorem}
This theorem holds for the classical method.
However as the hybrid method is only a data-parallel version of the classical
block method the same convergence rate hold with $k=p$ for the $\SubA[Hy]^p$ method.
The following lemma gives us a convergence rate for the global methods.
\begin{lemma}[Theoretical convergence rate of global methods]
  The theoretical convergence rate of a global method using $\SubA[Gl]^p$ is
  \begin{align}
    \hat{\mu} =  \frac{\hat{\kappa}_p-1}{\hat{\kappa}_p+1},
    \quad\text{with}\quad\hat{\kappa}_p = \frac{\lambda_N}{\lambda_{\left\lceil
        \frac{p}{q} \right\rceil}}
  \end{align}
\end{lemma}
\begin{proof}
    A global method is equivalent to solve the $qn$-dimensional system
  \begin{align}
    \begin{pmatrix}
      A\\
      & A\\
      && \ddots\\
      &&& A\\
    \end{pmatrix}
    \begin{pmatrix}
      X_1 & \cdots & X_p\\
      X_{p+1} & \cdots & X_{2p}\\
      &\vdots\\
      X_{k-p+1} & \cdots & X_k
    \end{pmatrix}
                           =
    \begin{pmatrix}
      B_1 & \cdots & B_p\\
      B_{p+1} & \cdots & B_{2p}\\
      &\vdots\\
      B_{k-p+1} & \cdots & B_k
    \end{pmatrix}
  \end{align}
  with the classical block Krylov method with $p$ right hand sides.
  The matrix of this system has the same eigenvalues as $A$ but with $q$ times
  the multiplicity.
  Thus the $p$-smallest eigenvalue is $\lambda_{\left\lceil \frac{p}{q} \right\rceil}$.
  Therefore and by applying Theorem \ref{thm:blockcg} we deduce the theoretical
  convergence rate.
\end{proof}
This result makes the global methods irrelevant for practical use.
In particular for $q>1$ the non-global hybrid method would perform better.
\section{Implementation and numerical experiments}
\label{sec:impl_and_numeric}
With the \dune 2.6 release an abstraction for SIMD data types was introduced.
The aim of these abstraction is the possibility to use the data types as a
replacement for the numeric data type, like \lstinline{double} or
\lstinline{float}, to create data parallel methods.
For a more detailed review see \citeauthor{bastian2019dune}~\cite{bastian2019dune}.
Usually these SIMD data types are provided by external libraries like
Vc~\cite{kretz2012vc} or vectorclass~\cite{fog2013c++}, which usually provide
data types with the length of the hardware SIMD (e.g.\ $4$ or $8$).
For problems with more right hand sides we use the \lstinline{LoopSIMD}
type.
This type is basically an array but implements the arithmetic operations by static sized loops.

\begin{lstlisting}[float,caption={Implementation of \texttt{BAXPY}}, label={lst:baxpy}]
void baxpy(scalar_field_type alpha,
           const BlockProduct<scalar_field_type>& gamma,
           const X& x, X& y){
  for(size_t i=0;i<x.size();++i){
    field_type xi = x[i]
    field_type yi = y[i];
    for(size_t r=0;r<P;++r){
      yi += lane(r, xi)*gamma[r];
    }
    y[i] = yi;
  }
}
\end{lstlisting}
Listing \ref{lst:baxpy} shows
the implementation for the \texttt{BAXPY} kernel
for the case $p=k$.

In a first test series we examine the runtime of the kernels \texttt{BDOT},
\texttt{BAXPY} and \texttt{BOP}.
To check the efficiency of our implementation we, take the following performance
model into account.
This performance model is a simplified variant of the ECM model presented by
\citeauthor{hofmann2019bridging}~\cite{hofmann2019bridging}.
We assume that the runtime of a kernel is bounded by the following three
factors:
\begin{itemize}
\item $\T{comp} = \frac{\omega}{\text{peakflops}}$: The time the processor needs
  to perform the necessary number of floating point operations.
  Where $\omega$ is the number of floating point operations of the kernel.
\item $\T{mem} = \frac{\beta}{\text{memory bandwidth}}$: The time to transfer
  the data from the main memory to the L1 cache.
  Where $\beta$ is the amount of data that needs to be transferred in the kernel.
\item $\T{reg} = \frac{\tau}{\text{register bandwidth}}$: The time to transfer
  data between L1 cache and registers.
  Where $\tau$ is the amount of data that needs to be transferred in the kernel.
\end{itemize}
Finally the expected runtime is given by
\begin{align}
  T = \max\left( \T{comp}, \T{mem}, \T{reg} \right).
\end{align}
Table \ref{tbl:kernels} shows an overview of the performance relevant
characteristics of the kernels.
We observe that for the \texttt{BOP} kernel the runtime per
right hand side decreases rapidly for small $k$, this is in accordance with our
expectation. For larger $k$ the runtime per right hand side increases slightly.
We suppose that this effect is due to the fact that for larger $k$ one row of a
block vector occupies more space in the caches, hence fewer rows can be cached.
This effects could be mitigated by using a different sparse matrix format, like
the Sell-C-$\sigma$ format~\cite{Kreutzer_2014}.

Furthermore we see that the runtime of the \texttt{BDOT} and \texttt{BAXPY}
kernels is constant up to an certain $p$ ($p\lesssim 16$).
This is in accordance with our expectation, as it is memory bound in that regime
and $\beta$ does not depend on $p$.
At almost all $p$ the runtimes for global and hybrid version coincide except for
$p=64$.
The reason for that is, that a $64\times 64$ takes $\SI{32}{\kilo\byte}$ memory,
which is exactly the L1 cache size.
In the non-global setting two of these matrices are modified during the
computation, which then exceeds the L1 cache.
This explains as well why the runtime for the $p=128$ case is so much higher
than expected.

Figure \ref{fig:microbenchmarks} shows the measured data compared with the
expected runtimes.
All tests are performed on an Intel Skylake-SP Xeon Gold 6148 on one core.
The theoretical peakflops are $\SI{76.8}{\giga\flop/\second}$, the memory bandwidth is
$\SI{13.345}{\giga\byte/\second}$ and the register bandwidth is $\SI{286.1}{\giga\byte/\second}$.
\begin{table}[t]
  \centering
  \begin{tabular}{r||c|c|c}
    & $\omega$ & $\beta$ & $\tau$\\\hline
    \texttt{BDOT}
    & $2np^2q$ & $2nk$ & $2nqp^2+2nk$ \\
    \texttt{BAXPY}
    & $2np^2q$ & $3nk$ & $nqp^2 + 2nk$\\
    \texttt{BOP}
    & $2kz$ & $2z + 2kn$ & $z(2+2k)$\\
  \end{tabular}
  \caption{Performance relevant characteristics for the \texttt{BDOT} and
    \texttt{BAXPY} kernels.
    Number of floating point operations ($\omega$), amount of data loaded from
    main memory ($\beta$), number of data transfers between registers and
    L1-Cache ($\tau$).
    $z$ is the number of non-zeros in $A$.
  }
  \label{tbl:kernels}
\end{table}
\begin{figure}[t]
  \centering
  \includegraphics[width=0.325\textwidth]{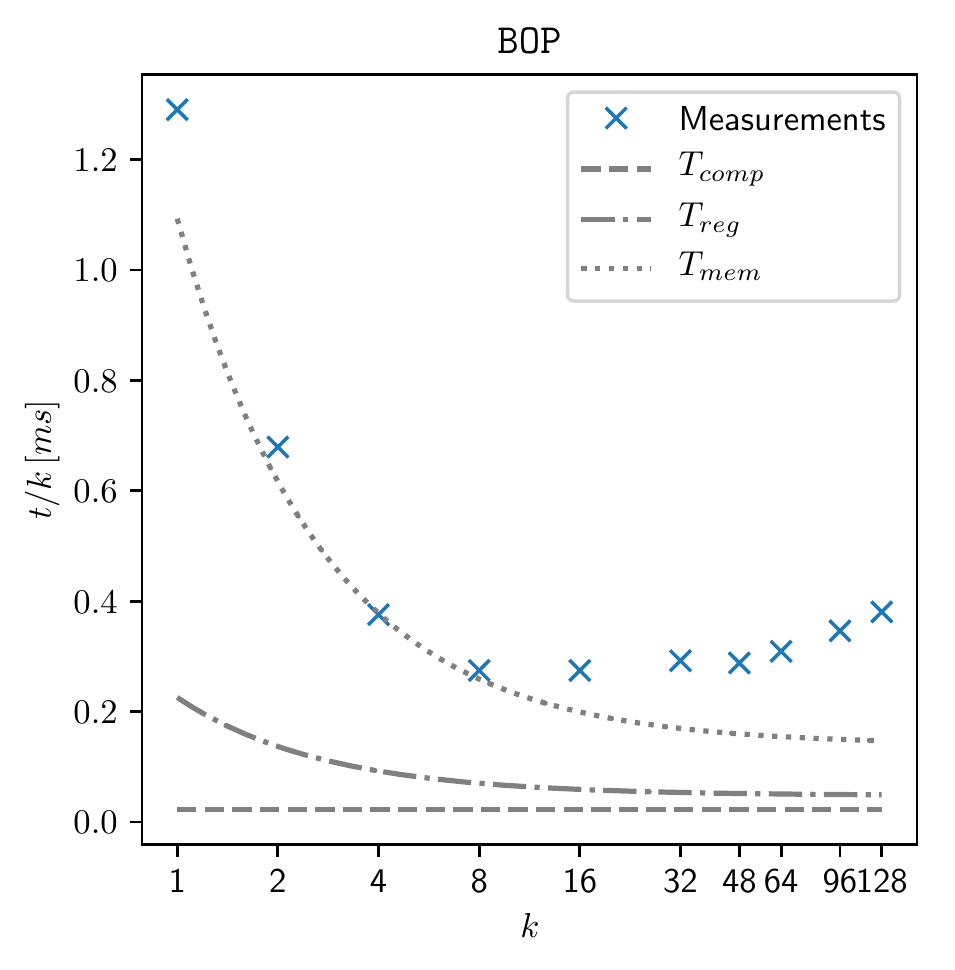}
  \includegraphics[width=0.325\textwidth]{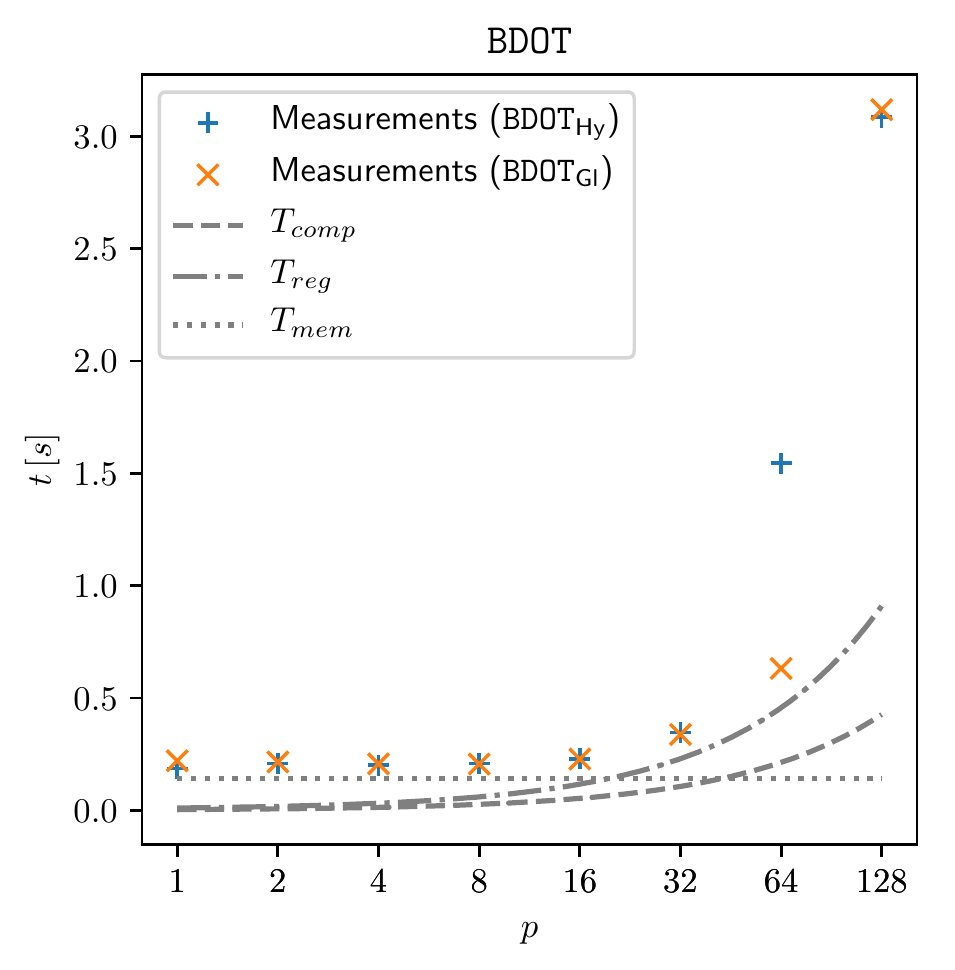}
  \includegraphics[width=0.325\textwidth]{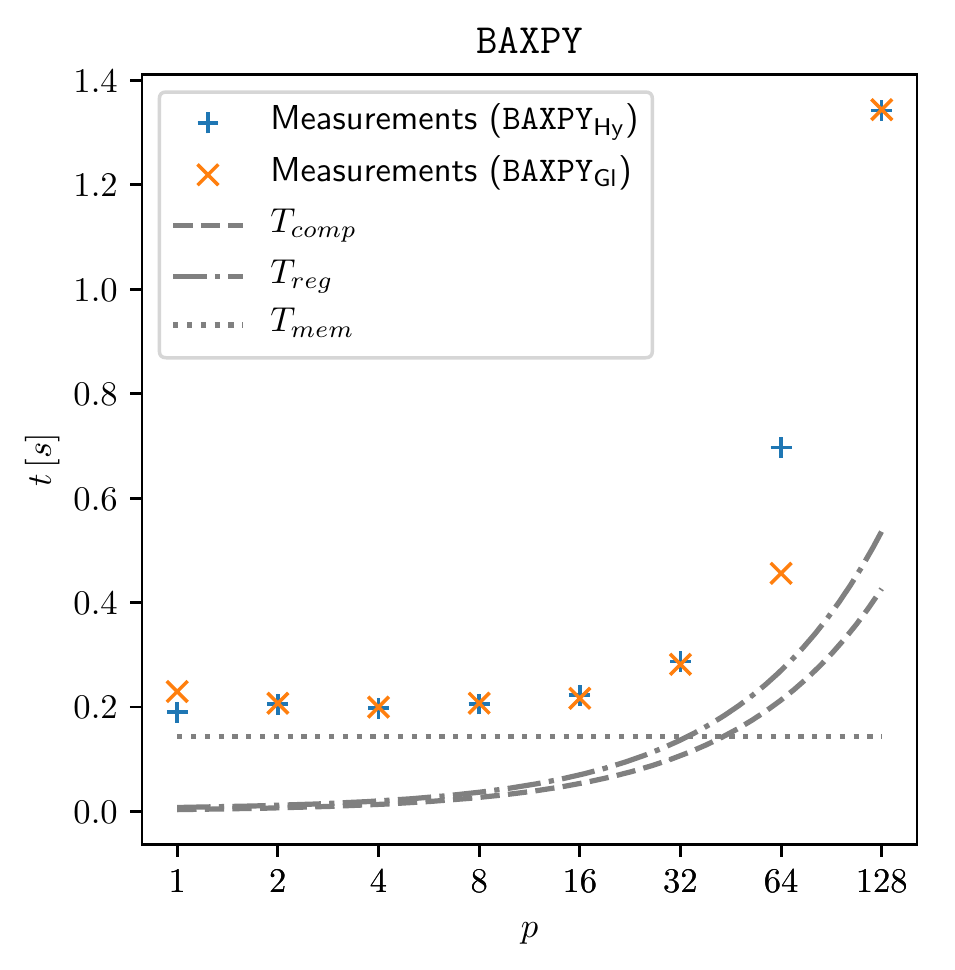}
  \caption{Microbenchmarks for kernels \texttt{BDOT}, \texttt{BAXPY} and \texttt{BOP} using $k=128$}
  \label{fig:microbenchmarks}
\end{figure}

In a second experiment we compare the runtime of the whole algorithm with each
other.
For that we discretized a 2D heterogeneous Poisson problem with a $5$-point Finite Difference
stencil.
The right hand sides are initialized with random numbers.
We iterate until the defect norm of each column has been decreased by a factor
of $10^{-8}$.
An ILU preconditioner was used.
Figure \ref{fig:bcg_comparison} shows the results.
We see that the best block size is $p=16$.
In another test we compare the runtimes for different parameters, where the
algorithm is executed $r$ times until all $128$ right hand sides are
solved.
In this case the $k=16,p=16$ case is the fastest but only slightly slower as the
$k=128,p=16$.
The reason for that is the worse cache behavior of the $\texttt{BOP}$ kernel,
like we have seen before.
Note that on a distributed machine the latter case would need $8$x less
communication.
\begin{figure}[t]
  \centering
  \includegraphics[width=0.325\textwidth]{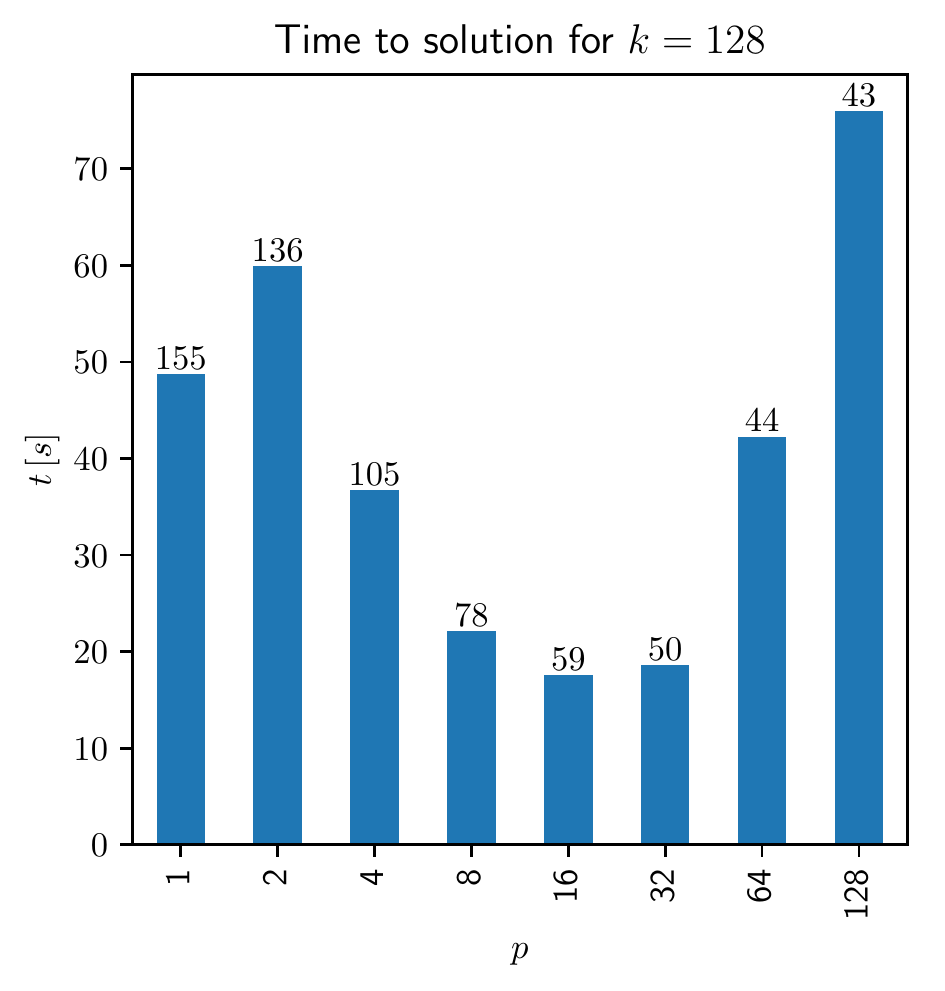}\qquad
  \includegraphics[width=0.325\textwidth]{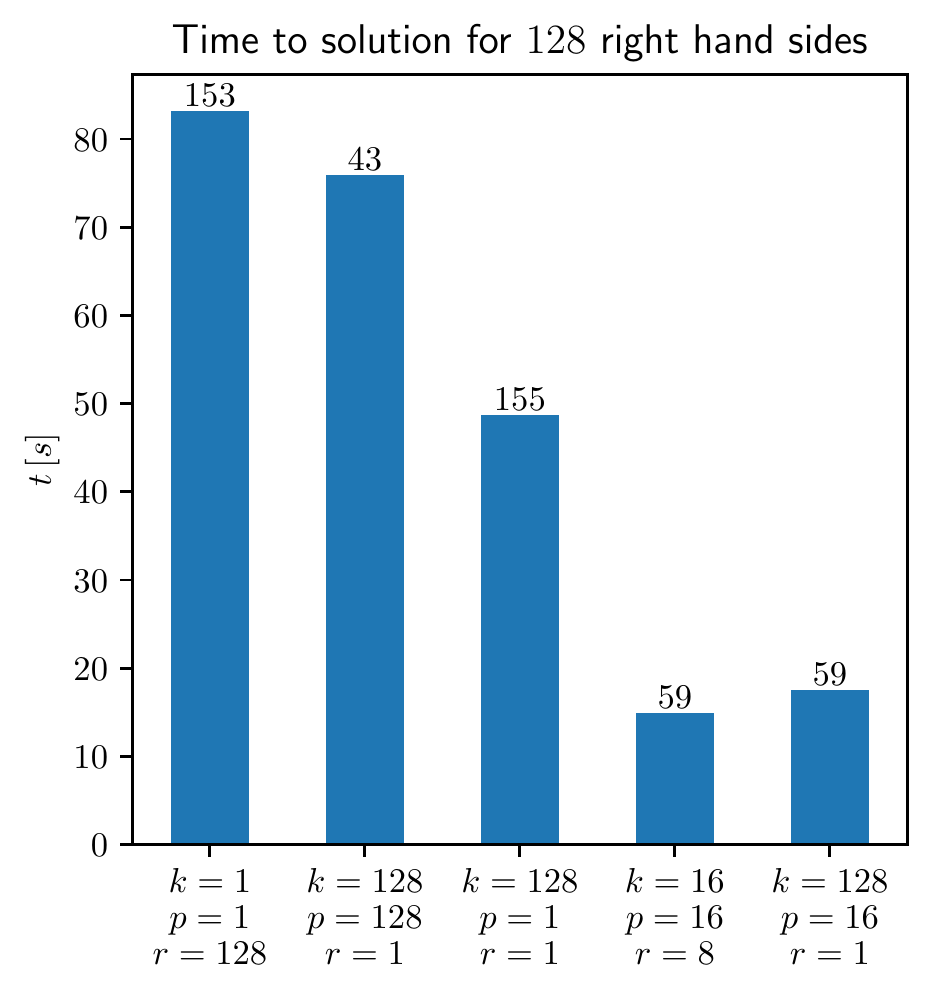}
  \caption{Time to solution for different parameters. Numbers on top of the bars denote the number of iterations. Left: $k=128$ Right: Different configurations: $r=128/k$ is the number of repetitions to solve for all $128$ right hand sides.}
  \label{fig:bcg_comparison}
\end{figure}
\section{Conclusion and outlook}
\label{sec:conclusion}
In this paper we have presented strategies for the vectorized BCG method.
For that we reviewed the block Krylov framework of
\citeauthor{frommer2017block} and apply it on the BCG method.
This makes it possible to use the advantages of the BCG method as far
as it is beneficial, while the number of right hand sides can be further
increased.
This helps to decrease the communication overhead and improve the
arithmetic intensity of the kernels.
We observed that the runtime of the individual kernels scale linearly with the number of right hand
sides as long as they are memory bound ($p \lesssim 16$ on our machine).
That means that it is always beneficial to use at least this block size $p$,
depending on the problem it could also be beneficial to choose even larger $p$.

The found optimizations are also applicable to other block Krylov space methods
like GMRes, MINRES or BiCG, and could be combined with pipelining techniques.
These approaches are the objective of future work.

\sloppy
\bibliographystyle{styles/spmpscinat}
\bibliography{simd_block_cg}
\end{document}